\documentclass[12pt]{amsart}
\input{header.sty} 

\begin{document}

\title[On $F$-pure thresholds and quasi-$F$-purity of hypersurfaces]{On $\mathbf{F}$-pure thresholds and quasi-$\mathbf{F}$-purity of hypersurfaces}
\begin{abstract}
    We show that quasi-$F$-pure but not $F$-pure isolated quasi-homogeneous hypersurface singularities necessarily have $F$-pure threshold $1 - \frac{1}{p}$. This extends work of Bhatt and Singh \cite{BhattSinghThresholds} beyond the Calabi-Yau case. We also classify the (quasi)-$F$-purity of Fermat hypersurfaces. 
\end{abstract}

\author[Garzella]{Jack J Garzella}
\address{Department of Matheamtics, UC San Diego, La Jolla, CA, USA}
\email{jgarzell@ucsd.edu}

\author[Jagathese]{Vignesh Jagathese}
\address{Department of Mathematics, Statistics, and Computer Science, University of Illinois at Chicago, Chicago, IL, USA}
\email{vjagat2@uic.edu}

\maketitle

\section{Introduction}

The goal of this article is to relate two different methods of measuring the failure of $F$-purity of a hypersurface in positive characteristic: the $F$-pure threshold and the quasi-$F$-pure height. In particular, we show via the $F$-pure threshold (denoted $\fpt(f)$) that quasi-$F$-pure hypersurfaces are nearly $F$-pure in a precise sense.

\begin{MAINTHM}[Theorem \ref{thm: MAIN FPT/QFP}]
    Let $f \in A = k[x_1, \dots, x_n]$ define a quasi-homogeneous isolated singularity,
    where \(k\) is a perfect field of odd characteristic \(p > n-2\).
    If \(A/(f)\) is quasi-\(F\)-pure, then
    \[
    \fpt(f) \geq 1 - \frac{1}{p} 
    .\]
\end{MAINTHM}

Introduced by Mustaţă, Takagi and Watanabe \cite{TakagiWatanabe_OnFPureThresholds, MustataTakagiWatanabeFThresholdsAndBernsteinSato}, The $F$-pure threshold of a hypersurface has proven to be a powerful invariant in positive characteristic geometry. For $f \in k[x_1, \dots, x_n]$, where $k$ is a perfect field of characteristic $p > 0$, $\fpt(f)$ is always a rational number \cite{BlickleMustataSmithDiscretenessAndRationalityOfFThresholds} contained in $(0,1]$ and $\fpt(f) = 1$ if and only if the corresponding quotient singularity $k[x_1, \dots, x_n]/(f)$ is $F$-pure. The lower the $F$-pure threshold is, the further away the hypersurface is from being $F$-pure. The $F$-pure threshold has been computed explicitly for wide classes of examples \cite{MustataTakagiWatanabeFThresholdsAndBernsteinSato, BudurMustataSaitoBernsteinSatoPolynomialsOfArbitrary,
HunekeMustataTakagiWatanabeFThresholdsTightClosureIntClosureMultBounds,
BlickleMustataSmithFThresholdsOfHypersurfaces,
HernandezThesis,
HernandezFPureThresholdOfBinomial,
HernandezFInvariantsOfDiagonalHyp, DeStefaniNunezBetancourtFthresholdGradedrings, 
BhattSinghThresholds, HernandezBetancourtWittEmilyZhang_FPureThresholdsofHomogeneousPolynomials,
KadyrsizovaKenkelPageSinghSmithVraciuWittLowerBoundsonFPureThresholdsandExtremalSingularities} and specific examples can be easily computed via the \verb|FrobeniusThresholds| package \cite{FrobeniusThresholdsMC2} in Macaulay2. \\ 

More recently, Yobuko introduced the notion of quasi-$F$-split singularities \cite{Yobuko_QuasiFrobeniusSplittingandliftingofCYVarsInCharp,Yobuko_OnTheFrobeniusSplittingHeightofVarsinPositiveChar} as a generalization of $F$-splitting.
Rather than requiring the absolute Frobenius $F: \Ox \to F_*\Ox$ split, quasi-$F$-split varieties instead only require the map $\Ox \to F_*(W_n\Ox/p)$ to split for some $n$, where $F_*(W_n\Ox/p)$ is constructed using Witt Vectors.
As $F_*W_1\Ox = F_*\Ox$, this requirement recovers the condition for $F$-splitting, and it weakens as $n$ grows. 
The minimal $n$ for which  $\Ox \to F_*(W_n\Ox/p)$ splits,
is called the \textbf{quasi-$\mathbf{F}$-split height}.
Initially introduced as a global invariant with a view towards Calabi-Yau varieties, quasi-$F$-splitting has proven to be a useful invariant in arithmetic and birational geometry \cite{NakkajimaYobukoKodaira,Kawakami++_2022quasifsplittingsbirationalgeometry,Kawakami++_2022quasifsplittingsbirationalgeometryII,Kawakami++_2022quasifsplittingsbirationalgeometryIII,tanaka2024quasifesplittings}. Quasi-$F$-purity is a local analog of quasi-$F$-splitting; they are equivalent in the setting of this paper \cite{Jagathese_OnQuasiFPurityofExcellentRings} so we use these terms interchangeably.  \\

 In \cite{BhattSinghThresholds}, Bhatt and Singh calculate the F-pure threshold of the affine cone of a smooth Calabi-Yau hypersurface $X = Z(f) \subset \PP^n_k$. If the characteristic of $k$ is strictly greater than $n^2 - n - 1$, they show that 
\[
\fpt(f) = 1 - \frac{a}{p}
\]
where $a = a(X)$ is the $a$-number of a variety in positive characteristic, a generalization of the $a$-number of an abelian variety, defined in full generality by van der Geer and Katsura \cite{vdG-Katsura-2002-invariant}. Van der Geer and Katsura also show that the $a$-number has the following relationship with the so-called Artin-Mazur height $h(X)$ of the hypersurface $f$: if $1 < h(X) < \infty$, then $a(X) = 1$, and if $1 = h(X)$, then $a(X) = 0$. 
Finally, Yobuko in \cite{Yobuko_QuasiFrobeniusSplittingandliftingofCYVarsInCharp} shows that the Artin-Mazur height is equal to the quasi-$F$-split height in this case. Taken together, these statements in the literature prove that, if $\textrm{char}(k) > n^2 - n - 1$ and $X = Z(f) \subset \PP^n_k$ is a smooth quasi-$F$-split Calabi-Yau hypersurface, then
\[
\fpt(f) \geq 1 - \frac{1}{p}
\] 
Via \cite{HernandezBetancourtWittEmilyZhang_FPureThresholdsofHomogeneousPolynomials}, \(1 - 1/p\) is the highest
\(F\)-pure threshold that can occur that is not exactly \(1\) in this setting.  Thus, this statement shows that a quasi-\(F\)-pure Calabi-Yau hypersurface is as close to being \(F\)-pure as possible: either it is just $F$-pure and $\fpt(f) = 1$ or it is not and $\fpt(f) = 1 - 1/p$ is maximal. \\ 

We extend this result to any quasi-homogeneous hypersurface, with a lesser restriction on \(p\); see Theorem \ref{thm: MAIN FPT/QFP} for a more precise statement. Our proof uses only the Fedder-type criterion for quasi-\(F\)-purity described in \cite{kawakami2022fedder} (cf. \cite{yoshikawa2025feddertypecriterionquasifesplittingquasifregularity}) and statements on the discreteness of $F$-pure Thresholds in \cite{HernandezBetancourtWittEmilyZhang_FPureThresholdsofHomogeneousPolynomials}, making
no mention of formal groups, deformation theory, or the Hodge and conjugate filtrations. \\


We caution that this is not an equivalence and provide examples of K3 surfaces which are supersingular (and thus not quasi-$F$-pure), yet have $a$-number $1$. We also apply the quasi-Fedder's criterion to detect (quasi)-$F$-purity of Fermat type hypersurfaces (Theorem \ref{thm: FermatHypersurfaceClassification}) and determine when $F$-purity and quasi-$F$-purity coincide (corollary \ref{cor: FermatHypQFPequivtoFP}).

\subsection*{Acknowledgments} The authors thank Kiran Kedlaya, Lance Miller, Kevin Tucker, Christopher Wong, and Shou Yoshikawa for valuable discussions and comments when preparing the first draft of this paper. The authors also thank the anonymous referee for their comments and suggestions. \\

The first author was partially supported by the National Science Foundation Graduate Research Fellowship Program under Grant No. 2038238, and a fellowship from the Sloan Foundation. The second author was partially supported by the National Science Foundation RTG under Grant No. 2037569.

\subsection*{AI Disclosure} The authors did not use artificial intelligence in any form during the production of this manuscript.

\section{Preliminaries}
In this section we will provide a brief overview of $F$-purity and $F$-pure thresholds. We will then define Witt Vectors with a view towards introducing quasi-$F$-purity. 
\subsection*{Notation}
\begin{itemize}
    \item All rings are assumed to be commutative, Noetherian, and contain $1$. Associated to any ring $R$ of characteristic $p > 0$, for $p$ a prime, will be the Frobenius morphism $F: R \to F_*R$ assigning $x \mapsto x^p$. $F_*(-)$ denotes restriction of scalars. In this paper we will further assume that $F$ is a finite map of $R$-modules (e.g. every characteristic $p$ ring $R$ is $\mathbf{F}$-\textbf{finite}). 
    \item $k$ will denote an $F$-finite field of characteristic $p > 0$ unless otherwise stated.
    \item In a local ring $R$, $\fm$ will denote the unique maximal ideal. Over any polynomial ring $k[x_1, \dots, x_n]$, $\fm := (x_1, \dots x_n)$ will be the maximal ideal corresponding to the origin. 
    \item For any $f \in k[x_1, \dots, x_n]$, $\fpt(f)$ denotes the $F$-pure threshold of $f$ and $\height(f)$ denotes the quasi-$F$-pure height of $k[x_1, \dots, x_n]/(f)$. See sections \ref{Preliminaries-FPTs} and \ref{Preliminaries-QuasiFPurity} for definitions of these constructions. 
    \item A polynomial $f \in k[x_1, \dots, x_n]$ is said to be \textbf{quasi-homogeneous} if it is homogeneous with respect to any grading of $k[x_1, \dots, x_n]$, not necessarily the standard grading. 
\end{itemize}
Let $R$ be a ring with hypotheses as above. Any map of $R$-Modules $M \to N$ is \textbf{pure} if $L \tensor_R M \to L \tensor_R N$ is injective for any $R$-Module $L$. We say that $R$ itself is $\mathbf{F}$-\textbf{Pure} (resp. $\mathbf{F}$-\textbf{split}) if the associated Frobenius map $F: R \to F_*R$ is pure (resp. is split) as an $R$-module homomorphism. We record some well known facts about $F$-purity and $F$-splitting below: 
\begin{itemize}
    \item $F$-split always implies $F$-pure, with equivalence when $R$ is complete local or $F$-finite.\footnote{Though equivalent in our ($F$-finite) setting, we will refer to rings in this paper as being $F$-pure. This is because $F$-splitting in more general contexts is not a good notion of singularity; indeed over non-$F$-finite fields there exist excellent DVRs which are not $F$-split \cite{DattaMurayamaTate}}
    \item $F$-purity is the characteristic $p$ analogue of having log canonical singularities \cite{HaraWatanabeFRegFPure}, and should be viewed as a 'mild' singularity condition. Indeed, $F$-pure rings are log canonical, and rings with log canonical singularities are conjecturally $F$-pure mod $p$ for infinitely many primes $p$. 
    \item $F$-purity is a local property, completes, and is stable under \'etale extensions and arbitrary base change.
\end{itemize}
In view of the last point, we can reduce to case where $R$ is complete local when testing for $F$-purity; in particular $R$ can be written as a quotient $R = A/I$ of a regular complete local ring $A$. Here $F$-purity can be checked explicitly by Fedder's Criterion:
\begin{theorem}[Fedder's Criterion \cite{FeddersCrit}]
   Let $(A,\fm)$ be a regular local ring and $I \subset A$ an ideal. Then $A/I$ is $F$-pure if and only if $(I^{[p]}:I) \not\subset \fm^{[p]}$.
\end{theorem}
When working with a hypersurface (e.g. $I = (f)$ for some $f \in A$) this criterion is even simpler: $A/(f)$ is $F$-pure if and only if $f^{p-1} \not\in \fm^{[p]}$. If $A/(f)$ is not $F$-pure, one can ask how close it is to being $F$-pure.
\subsection{\TOC{$\mathbf{F}$}{F}-Pure Thresholds}\label{Preliminaries-FPTs}
Let $A = k[x_1, \dots, x_n]$. For a polynomial $f \in \fm$ and any $e \geq 1$, we define 
$$\nu_f(p^e) := \max \left\{N \ | \ f^N \not\in \fm^{[p^e]}\right\}$$
Via Fedder's criterion above, if $R:= A/(f)$ is $F$-pure, it can be readily seen that $\nu_f(p) = p - 1$. More generally, for $e \gg 0$, $\nu_f(p^e) \sim p^e$. If $R$ was not $F$-pure, $\nu_f(p^e)$ would be smaller than $p^e$ as $e$ grows. The ratio between these quantities in the limit is precisely the $\mathbf{F}$-\textbf{pure threshold}:
$$\fpt(f) := \lim_{e \to \infty}\frac{\nu_f(p^e)}{p^e} \in (0,1]$$
Though initially introduced in \cite{TakagiWatanabe_OnFPureThresholds} for pairs $(R,\fa)$, this special case first appeared in \cite{MustataTakagiWatanabeFThresholdsAndBernsteinSato} and is all we will need for this paper. We note that $\left\{\frac{\nu_f(p^e)}{p^e}\right\}_e$ forms a non-decreasing sequence, so for any $e \geq 1$,
$$\fpt(f) \geq \frac{\nu_f(p^e)}{p^e} > 0$$
 Importantly $\fpt(f) = 1$ if and only if $R$ is $F$-pure, and heuristically, the lower $\fpt(f)$ is the worse the singularities are at the origin. Just as $F$-pure singularities are related to log canonical singularities in characteristic $0$, there exists an analogous relationship between the $F$-pure threshold and log canonical threshold:
\begin{theorem}[\cite{MustataTakagiWatanabeFThresholdsAndBernsteinSato} Theorem 3.3,3.4]
    Let $f \in \CC[x_1, \dots, x_n]$ be a polynomial with integer coefficients. For a given prime $p$, let $f_p \in \FF_p[x_1, \dots, x_n]$ denote $f$ with coefficients reduced modulo $p$. Then for all primes $p$, $\fpt(f_p) \leq \mathrm{lct}(f)$, and further,
    $$\lim_{p \to \infty} \fpt(f_p) = \mathrm{lct}(f)$$
\end{theorem}

\subsection{Witt Vectors}
In order to define quasi-$F$-purity, we first provide a very brief overview of Witt Vectors. The authors stress that readers not familiar with Witt Vectors can just take the Fedder's Type Criterion of \cite{kawakami2022fedder}, which will be stated in the following section, to be the definition of quasi-$F$-purity for simplicity. We will not explicitly need to use properties of Witt Vectors for any proofs in this paper. For a more thorough treatment of Witt Vectors, the authors recommend \cite{Borger}, \cite[Appendix]{LangerZink}, \cite{DavisKedlaya}, and \cite{LenstraWittVectors}. \\

To any commutative ring $R$ we can associate a ring of ($p$-typical) \textbf{Witt Vectors} 
$$W(R) = \{(\alpha_0,\alpha_1, \dots, ) \ | \ \alpha_i \in R\}$$
We will typically denote elements of $W(R)$ as $\alpha = (\alpha_0,\alpha_1, \dots)$. For any $r \in R$ the element $[r] := (r,0,\dots) \in W(R)$ denotes the \textbf{lift} of $r$ to $W(R)$. Attached to $W(R)$ are addition and multiplication operations defined via universal Witt Polynomials $S_\bullet(\alpha,\beta)$ and $P_\bullet(\alpha,\beta)$ such that
 $$\alpha + \beta = (S_0(\alpha_0,\beta_0), S_1(\alpha_0,\alpha_1,\beta_0,\beta_1), \dots, S_n(\alpha_{\leq n}, \beta_{\leq n}),\dots )$$
 $$\alpha \cdot \beta = (P_0(\alpha_0,\beta_0), P_1(\alpha_0,\alpha_1,\beta_0,\beta_1), \dots, P_n(\alpha_{\leq n}, \beta_{\leq n}),\dots )$$
$S_0(\alpha_0,\beta_0) = \alpha_0 + \beta_0$ and $P_0(\alpha_0,\beta_0) = \alpha_0\beta_0$ as expected, but higher order Witt polynomials are far more complex. As these are polynomial relations, one easily sees that $W(-)$ is functorial: for any ring morphism $\varphi: R \to S$ one can define a ring morphism $W(\varphi): W(R) \to W(S)$ assigning $\alpha \mapsto (\varphi(\alpha_0),\varphi(\alpha_1), \dots)$. $W(R)$ has the following associated maps: 
\begin{itemize}
    \item \textbf{Frobenius}: Witt rings admit a Witt-Frobenius $F: W(R) \to W(R)$. When $R$ is characteristic $p > 0$, this is simply the image of $F:R \to F_*R$ under the functor $W(-)$ assigning $F(\alpha) = (\alpha_0^p,\alpha_1^p,\dots)$. When $R$ is $F$-finite and Noetherian, the Witt Frobenius is a finite map of $W_n(R)$-modules \cite[lemma 2.5]{kawakami2022fedder}.\footnote{The authors stress that the Witt Frobenius is not the same as the ''standard'' Frobenius $\alpha \mapsto \alpha^p$, which is not a ring map on $W(R)$.}
    \item \textbf{Verschiebung}: $V: W(R) \to W(R)$ assigning $(\alpha_0,\alpha_1, \dots) \mapsto (0,\alpha_0,\alpha_1, \dots)$. 
\end{itemize}
We define $W_n(R) := W(R)/\im(V^n)$ to be the ring of \textbf{$\mathbf{n}$-truncated $\mathbf{p}$-typical Witt Vectors}, with similar additional and multiplication operations and associated maps $V,F,[-]$. One easily sees from this definition (along with the definitions of $S_0,P_0$) that $W_1(R) \isom R$. These also have an associated restriction map  $\fR: W_{n+1}(R) \to W_{n}(R)$ assigning $(\alpha_0, \dots, \alpha_n) \mapsto (\alpha_0, \dots, \alpha_{n-1})$. $F$ and $\fR$ are ring homomorphisms while $V$ is additive but not multiplicative. We note $p \in W_n(R)$ is of the form $(0,1,0, \dots 0)$ and the map $p: \alpha \mapsto p \cdot \alpha$, when $R$ is characteristic $p > 0$, yields the identity $p = FV = VF$. \\ 

Let $\ov{W_n}(R) := W_n(R)/p$ be the mod $p$ reduction of $W_n(R)$. $F_*\ov{W_n}(R)$ is an $R$-module via the action $r \cdot F_*\ov{\alpha} := F_*\ov{\left([r^p] \cdot \alpha\right)}$, and is a $W_n(R)$-Module in the expected way. $R$ is also a $W_n(R)$ module via the restriction map $W_n(R) \xto{\fR^{n-1}} R$. \\

\subsection{Quasi-\TOC{$\mathbf{F}$}{F}-purity}\label{Preliminaries-QuasiFPurity}
From the previous section, we can see that the Frobenius $R \to F_*R$ factors through $F_*\ov{W_n}(R)$ for any $n$, yielding the following sequence of $R$-Modules:
$$R \xto{r \mapsto F_*\ov{(r^p,0, \dots, 0)}} F_*\ov{W_n}(R) \xto{\fR} \dots \xto{\fR} F_*\ov{W_2}(R) \xto{\fR}F_*\ov{W_1}(R) = F_*R$$
From this diagram, on easily sees that if $R \to F_*R$ is pure (resp. splits) then $R \to F_*\ov{W_n}(R)$ is pure (resp. splits) for all $n \geq 1$. Thus, the purity (resp. splitting) of $R \to F_*\ov{W_n}(R)$ can be seen as a weakening of $F$-purity (resp. $F$-splitting), with the condition weakening further as $n$ grows. This motivates the following definition: \\

We say that $R$ is \textbf{$\mathbf{n}$-quasi-$\mathbf{F}$-pure} (resp. \textbf{$\mathbf{n}$-quasi-$\mathbf{F}$-split}) if $R \to F_*\ov{W_n}(R)$ is a pure (resp. split) map of $R$-modules. If $R$ is $n$-quasi-$F$-pure (resp. $n$-quasi-$F$-split) for some $n$, then we say $R$ is quasi-$F$-pure (resp. quasi-$F$-split). Similarly to $F$-pure and $F$-split singularities, $n$-quasi-$F$-split implies $n$-quasi-$F$-pure, with equivalence when $R$ is complete local or $F$-finite \cite{Jagathese_OnQuasiFPurityofExcellentRings}. As we are in the $F$-finite setting, we can use these terms interchangeably. \\ 

$1$-quasi-$F$-purity is equivalent to $F$-purity, and from our discussion above, it is clear that $n$-quasi-$F$-pure implies $(n+1)$-quasi-$F$-pure. Thus we define $\height(R)$, the \textbf{quasi-$\mathbf{F}$-pure height}, to be the minimal $n$ for which $R$ is $n$-quasi-$F$-pure. $\height(R) = 1$ if and only if $R$ is $F$-pure, and $\height(R) = \infty$ if $R$ is not $n$-quasi-$F$-pure for any $n$. As we are particularly concerned with hypersurfaces, for $f \in A$ (with $A$ a regular ring) we denote $\height(f) := \height(A/(f))$ for simplicity. \\ 

Quasi-$F$-purity was inspired by (global) quasi-$F$-splittings introduced by Yobuko \cite{Yobuko_QuasiFrobeniusSplittingandliftingofCYVarsInCharp}. Quasi-$F$-splitting has proven to be a powerful invariant, detecting arithmetic properties like supersingularity \cite{Yobuko_QuasiFrobeniusSplittingandliftingofCYVarsInCharp,Yobuko_QuasiFSplitandHodgeWitt}, extending the Minimal Model Program in positive characteristic \cite{Kawakami++_2022quasifsplittingsbirationalgeometry,Kawakami++_2022quasifsplittingsbirationalgeometryII,Kawakami++_2022quasifsplittingsbirationalgeometryIII}, and governing degeneration of the Hodge-de-Rham spectral sequence \cite{Petrov_DecompositionofDeRhamComplexForQuasiFSplitVarieties}. Quasi-$F$-purity is the local analogue of quasi-$F$-splitting. Much like $F$-purity, for excellent local rings, quasi-$F$-purity completes, localizes, and is stable under local or finite \'etale extension \cite{Jagathese_OnQuasiFPurityofExcellentRings}. 

\section{Main Result}
In this section we develop the machinery necessary to state the Fedder's style criterion for quasi-$F$-purity in \cite{kawakami2022fedder}, similarly to section 1.1 in their paper. We then use this criterion to compute the $F$-pure threshold of a quasi-$F$-pure hypersurface singularity. \\ 

For simplicity, we take our regular ring $A$ to be the localized polynomial ring $k[x_1, \dots, x_n]_\fm$. As quasi-$F$-purity localizes and completes, this is not a significant simplification. $F_*A$ is a free $A$ module with a basis consisting of monomials $F_*x_1^{i_1}x_2^{i_2}\dots x_n^{i_n}$, where each $i_j \leq p-1$. Further, the dual space $\Hom_A(F_*A,A)$ is principally generated as a $F_*A$-module by the dual basis vector $u$ of $F_*(x_1x_2\dots x_n)^{p-1}$. Explicitly, 
$$\Hom_A(F_*A,A) \isom F_*A \cdot u \qquad u(F_*b) = \begin{cases} 1 & b = F_*(x_1x_2\dots x_n)^{p-1} \\ 0 & \textrm{ else} \end{cases}$$
Where $F_*b$ denotes an arbitrary $A$-basis element of $F_*A$, viewed as a free $A$-module. The $u$ map will be used in the statement of the criterion, as will a certain map $\Delta_1: A \to A$ which we will now define. For any element $a \in A$, we can write $a = \sum a_i M_i$ as a $k$-linear sum of monomials $M_i = \prod_{j=1}^n x_j^{i_j}$. In $W_2(A)$, $(a,0)$ and $\sum (a_i M_i,0)$ do not differ in the $0$th coordinate, but do differ in the 1st coordinate; $\Delta_1(a)$ measures this difference. In particular, for a given $a \in A$ we define $\Delta_1(a) \in A$ to be the unique value that satisfies the equation $(0,\Delta_1(a)) = (a,0) - \sum (a_iM_i,0)$. Using the universal Witt polynomials, one obtains a more explicit description: 
$$\Delta_1(a) := \frac{\left(\sum a_i M_i\right)^p - \sum (a_iM_i)^p}{p} = \sum_{0 \leq \alpha_1, \dots, \alpha_m \leq p-1, \sum \alpha_j = p}\frac{1}{p}\binom{p}{\alpha_1, \dots, \alpha_m}\prod_{i=1}^m(a_iM_i)^{\alpha_i}$$
As $p$ divides $\binom{p}{\alpha_1, \dots, \alpha_m}$ for any valid choice of $(\alpha_j)$, this division by $p$ is purely formal. This is all we will need to state the theorem for hypersurfaces $A/(f)$. The authors of \cite{kawakami2022fedder} actually obtained a more general criterion for complete intersections, see \cite[Theorem 4.11]{kawakami2022fedder}. We state the criterion for principal quotients of polynomial rings below.
\begin{theorem}[Fedder Type Criterion for Quasi-$F$-Purity \cite{kawakami2022fedder}] \label{thm: FeddersTypeCrit}
    Let $A = k[x_1, \dots, x_n]_\fm$ and pick $f \in A$. Define $\theta: F_*S \to S$ to be the map
    $$\theta(F_*a) := u\left(F_*\Delta_1(f^{p-1})a\right)$$
    Letting $I_1 := (f^{p-1})$, we define the increasing sequence of ideals $\{I_m\}_m$ inductively as follows:
    $$I_{m+1} := \theta\left(I_m \cap \ker(u)\right) + (f^{p-1})$$
    Then
    $$\height(f) = \inf\{m \ | \ I_m \not\subset \fm^{[p]}\}$$
\end{theorem}
We utilize the convention that the infimum of the empty set is $\infty$. If $f$ is homogeneous with respect to an $\NN$-grading on $A$, we can replace $A$ with $k[x_1, \dots, x_n]$ (i.e. we don't need to assume that $A$ is local, see \cite[corollary 5.4]{kawakami2022fedder}). This theorem recovers the standard Fedder's criterion and can detect the exact quasi-$F$-pure height of a given hypersurface. Though the statement looks cumbersome at first glance, it is suitable for implementation in a computer algebra system using GPU acceleration techniques (see \cite{BatubaraGarzellaPan_K3SurfacesofanyArtinMazurHeightoverF5andF7}, \cite{fgmqt-2025-witt-vectors-macaulay2}) and has been used to compute many examples of quasi-$F$-pure hypersurfaces. Using their theorem, the authors obtained a useful criterion for determining when a hypersurface is not non-Quasi-$F$-pure.
\begin{corollary}[\cite{kawakami2022fedder}, corollary 4.19]
\label{cor: QuasiFedderCriterionNotQFPure}
    Let $A = k[x_1, \dots, x_n]_\fm$ or $A = k[x_1, \dots, x_n]$ and $f \in A$. 
    \begin{enumerate}[label = (\alph*)]
        \item If $f^{p-2} \in \fm^{[p]}$, then $\height(f) = \infty$.
        \item If $f^{(p+1)(p-2)}\Delta_1(f) \in \fm^{[p^2]}$ for some representative $\Delta_1(f) \in A$ and $f^{p-1} \in \fm^{[p]}$, then $\height(f) = \infty$. 
    \end{enumerate}
\end{corollary}
We note that \cite[corollary 4.19]{kawakami2022fedder} only proves the above statement for $A = k[x_1, \dots, x_n]_\fm$, but it also holds for $A = k[x_1, \dots, x_n]$ by considering \cite[proposition 2.25]{kawakami2022fedder}. Corollary \ref{cor: QuasiFedderCriterionNotQFPure}(a) immediately implies that, if $A/(f)$ is quasi-$F$-pure, $f^{p-2} \not\in \fm^{[p]}$. This yields the inequality $\fpt(f) \geq \nu_f(p)/p \geq 1 - \frac{2}{p}$. We will now use \ref{cor: QuasiFedderCriterionNotQFPure}(b) to obtain a sharper bound. 
\begin{theorem} \label{thm: MAIN FPT/QFP}
    Let $A = k[x_1, \dots, x_n]$ and $f \in A$ define an isolated hypersurface singularity or be quasi-homogeneous.  If $A/(f)$ is quasi-$F$-pure, then $\fpt(f) \geq 1 - \frac{1}{p} - \frac{2}{p^2}$. If we assume $f$ is an isolated hypersurface singularity AND quasi-homogeneous, and in addition that $p > n-2$, $p \neq 2$, and that $k$ is perfect, then $\fpt(f) \geq 1 - \frac{1}{p}$.
\end{theorem}
\begin{proof}
    If $f$ has an isolated singularity, via a change of coordinates we can assume the isolated singularity is at the origin. We can localize at $\fm = (x_1, \dots, x_n)$ and can thus assume that $A = k[x_1, \dots, x_n]_\fm$. If $f$ is quasi-homogeneous, we can reduce to the local case via \cite[proposition 2.25]{kawakami2022fedder}. If $\height(f) < \infty$, by the contrapositive of \ref{cor: QuasiFedderCriterionNotQFPure}(b) we see that either $f^{p-1} \not\in \fm^{[p]}$ (in which case $\fpt(f) = 1$ as $A/(f)$ is $F$-pure by Fedder's Criterion) or $f^{(p+1)(p-2)}\Delta_1(f) \not\in \fm^{[p^2]} $. The former case is immediate, so we reduce to the latter case. If $f^{(p+1)(p-2)}\Delta_1(f) \not\in \fm^{[p^2]} $, then $f^{(p+1)(p-2)} = f^{p^2 - p - 2} \not\in \fm^{[p^2]}$, and hence, 
    $$\fpt(f) \geq \frac{\nu_f(p^2)}{p^2} \geq1 - \frac{1}{p} - \frac{2}{p^2}$$
    If $f$ is an isolated singularity over a perfect field, then $\mathrm{Jac}(f) = \fm$. Assuming $f$ is quasi-homogeneous, we fix an $\NN$-grading on $A$ such that $f$ is homogeneous. In this setting, we note that if $\sum \deg(x_i) < \deg(f)$ then $A/(f)$ is of general type, and thus, never quasi-$F$-pure. Thus, if   $\sum \deg(x_i) \geq \deg(f)$ then $\lambda = a = b = 1$ for $\lambda,a,b$ used in \cite[Theorem 3.5]{HernandezBetancourtWittEmilyZhang_FPureThresholdsofHomogeneousPolynomials}. If $p > n-2$, then by \cite[Theorem 3.5(2)]{HernandezBetancourtWittEmilyZhang_FPureThresholdsofHomogeneousPolynomials} $\fpt(f)$ takes the form $1 - \frac{h}{p}$ for some $h$.  The bound $\fpt(f) \geq 1 - \frac{1}{p} - \frac{2}{p^2}$ shows that $h < 2$ provided that $p \neq 2$, so it follows that
    $$\fpt(f) \geq 1 - \frac{1}{p}$$ 
    as desired.
\end{proof}
As $\fpt(f) = 1$ if and only if $A/(f)$ is $F$-pure, we obtain the following equality.
\begin{corollary}
    Let $A= k[x_1, \dots, x_n]$ for $k$ a perfect field of odd characteristic $p > n - 2$ and $f \in A$ be a quasi-homogeneous isolated singularity. If $\height(f) \in (1,\infty)$ (e.g. $A/(f)$ is quasi-$F$-pure at the isolated singularity but not $F$-pure), then $\fpt(f) = 1 - \frac{1}{p}$.
\end{corollary}
We note that $1 - \frac{1}{p}$ is the highest attainable $F$-pure threshold for a quasi-homogeneous polynomial that is not $F$-pure by \cite[corollary 3.7]{HernandezBetancourtWittEmilyZhang_FPureThresholdsofHomogeneousPolynomials}. Thus, in terms of the $F$-pure threshold and in view of theorem \ref{thm: MAIN FPT/QFP}, quasi-$F$-pure singularities are ``as close to being $F$-pure'' as possible. We caution that the converse of theorem \ref{thm: MAIN FPT/QFP} is false in many cases. 
\begin{exam}
    The following equations in $\FF_7[x_1,x_2,x_3,x_4]$ define supersingular K3 surfaces (i.e. K3 surfaces that are not quasi-$F$-pure) which have $F$-pure threshold $1 - 1/7$. 
    \begin{itemize}
    \item $4x_1^4 + 5x_1^3x_3 + 3x_1^3x_4 + 6x_1^2x_2^2 + 2x_1^2x_2x_3 + 3x_1^2x_2x_4 + 3x_1^2x_3^2 + 3x_1^2x_3x_4 + 5x_1^2x_4^2 + 3x_1x_2^2x_4 + 6x_1x_2x_3^2 + x_1x_2x_4^2 + 2x_1x_3^2x_4 + 5x_1x_3x_4^2 + 4x_1x_4^3 + 5x_2^3x_3 + 5x_2^3x_4 + 5x_2^2x_3^2 + 2x_2^2x_3x_4 + 3x_2^2x_4^2 + 5x_2x_3^2x_4 + 4x_2x_3x_4^2 + 2x_2x_4^3 + 2x_3x_4^3$ \\
    \item $5x_1^4 + 3x_1^3x_2 + x_1^3x_3 + 6x_1^3x_4 + 5x_1^2x_2^2 + 4x_1^2x_2x_4 + 2x_1^2x_3^2 + 5x_1^2x_3x_4 + 4x_1x_2^3 + 4x_1x_2^2x_3 + 4x_1x_2^2x_4 + 5x_1x_2x_3x_4 + 2x_1x_3^2x_4 + 5x_1x_3x_4^2 + x_1x_4^3 + 5x_2^4 + x_2^3x_3 + 3x_2^2x_3^2 + 4x_2^2x_3x_4 + 4x_2^2x_4^2 + 3x_2x_3^3 + 5x_2x_3^2x_4 + 5x_2x_4^3 + 4x_3^4 + 5x_3^3x_4 + x_3x_4^3 + 5x_4^4$ \\
    \item $3x_1^4 + 3x_1^3x_2 + 3x_1^3x_3 + 6x_1^2x_2^2 + 3x_1^2x_2x_4 + 2x_1^2x_3^2 + 2x_1^2x_3x_4 + 3x_1^2x_4^2 + 6x_1x_2^3 + 5x_1x_2^2x_3 + x_1x_2x_3x_4 + 5x_1x_2x_4^2 + 5x_1x_3^3 + 4x_1x_3^2x_4 + 3x_1x_3x_4^2 + 6x_1x_4^3 + x_2^4 + 4x_2^3x_4 + 3x_2^2x_3^2 + 5x_2^2x_3x_4 + 5x_2x_3^3 + x_2x_3^2x_4 + 6x_2x_3x_4^2 + x_3^3x_4 + x_3^2x_4^2 + 3x_3x_4^3 + 4x_4^4$
\end{itemize}
\end{exam}

These examples were computed via sampling from a space of certain polynomials using the methods of \cite[Section~7]{BatubaraGarzellaPan_K3SurfacesofanyArtinMazurHeightoverF5andF7}, which itself uses \cite[Theorem~C]{kawakami2022fedder}. Then, Macaulay2 \cite{M2} was used to calculate their F-pure thresholds, see \cite{FrobeniusThresholdsMC2}.

\section{Fermat Type Hypersurfaces}
The goal of this section is to investigate when a Fermat type hypersurface (e.g. a hypersurface defined by an equation of the form $x_1^d + \dots + x_n^d$) is $F$-pure, quasi-$F$-pure but not $F$-pure, or not quasi-$F$-pure. We first note that the characteristic $2$ differs slightly from the general case: we outline this case below.
\begin{lemma}\label{lem: Fermat char2}
    For $n \geq 2$, set $f = x_1^{d} + \dots + x_n^{d}  \in k[x_1, \dots, x_n]$ over $k$, an $F$-finite field of characteristic $2$. Then
    $$\height(f) = \begin{cases}
        1 & d = 1 \\
        2 & d = 3 \textrm{ and } n \geq 3\\
        \infty & \textrm{ otherwise}
    \end{cases}$$
\end{lemma}
\begin{proof}
    The case of $d = 1$ is clear. If $d = 2$ or $d = 4$, $R$ is non-reduced and hence not quasi-$F$-pure. If $d = 3$, via \cite[Example 7.9]{kawakami2022fedder} we see that the Fermat cubic has height $2$ in characteristic $2$ when $n \geq 3$; if $n < 3$ then $R$ is of general type and hence not quasi-$F$-pure. This just leaves the case where $d \geq 5$. \\ 

    To prove that $\height(f) = \infty$ when $d \geq 5$, we can apply corollary \ref{cor: QuasiFedderCriterionNotQFPure}(b) when $p = 2$: It is thus sufficient to show that $\Delta_1(f) \in \fm^{[4]}$. Fortunately in characteristic $2$, $\Delta_1(f)$ is easily computable.
    $$\Delta_1(f) = \frac{\left(\sum_{i=1}^n x_i^{d}\right)^2 - \sum_{i=1}^n x_i^{2d}}{2} = \sum_{i < j}x_i^d\cdot x_j^d$$
    When $d \geq 5$ it immediately follows that $\Delta_1(f) \in \fm^{[4]}$.
\end{proof}
\begin{corollary}
    As this characterization holds for all $n \geq 2$, this lemma provides examples of smooth Fano varieties of any dimension $\geq 5$ in characteristic $2$ that are not quasi-$F$-pure.
\end{corollary}
In view of this we will primarily concern ourselves with the following setting in this section:
\begin{setting} \label{setting: FermatHypersurfaces char>2}
    Let $A = k[x_1, \dots, x_n]$ where $k$ is a perfect field of characteristic $p > 2$ and $n \geq 3$. Let $f = x_1^d + \dots x_n^d\in A$ for some $d > 1$ and take $R = A/(f)$.  
\end{setting}
In general, $F$-pure thresholds of any diagonal hypersurface have been computed in \cite{HernandezFInvariantsOfDiagonalHyp}; this can be used to classify whether a Fermat Hypersurface is $F$-pure or not. While these methods can partially recover the following results, it cannot detect for quasi-$F$-purity. Further, one can classify Fermat hypersurfaces very explicitly using the Fedder's criterion for (quasi)-$F$-purity \cite{kawakami2022fedder}: We spell these arguments out below.  
\begin{lemma}\label{lem: low degree}
In setting \ref{setting: FermatHypersurfaces char>2}, if $d < \lceil \frac{p}{a} \rceil$ and  $n \geq \frac{p-1}{a} $ for some $a \geq 1$, $R$ is $F$-pure.
\end{lemma}
\begin{proof}
    Via Fedder's criterion \cite{FeddersCrit}, we need to show that $f^{p-1} \not\in \fm^{[p]}$. It is sufficient to construct a monomial present in $f^{p-1}$ which contains no $p$th powers. A general monomial of $f^{p-1}$ is of the following form:
    $$c \cdot x_1^{i_1 \cdot d}x_2^{i_2 \cdot d} \dots x_n^{i_n \cdot d}$$
    For $0 \neq c \in k$ and $\sum i_j = p-1$. Our goal is to pick $(i_1, \dots i_n)$ such that $i_j \cdot d < p$ for all $j$. Well as $d < \left\lceil \frac{p}{a}\right\rceil $, this is equivalent to finding $(i_1, \dots i_n)$ such that $i_j  \leq a$. As $a\cdot n \geq p-1 = \sum i_j$, it is necessarily possible to choose the $i_j$ such that they are all at most $a$; simply set them all to $a$ then subtract ones from them arbitrarily (without making any $i_j$ negative) $a\cdot n - (p-1)$ times.
\end{proof} 
\begin{lemma}\label{lem: high degree}
In setting \ref{setting: FermatHypersurfaces char>2}, if $d \geq \left\lceil\frac{p}{a+1}\right\rceil$ and $n < \frac{p-1}{a}$, then $R$ is not $F$-pure. If we further assume $n < \frac{p-2}{a}$, then $R$ is not quasi-$F$-pure. 
\end{lemma}
\begin{proof}
  Assume that $n < \frac{p-1}{a}$. Via Fedder's criterion \cite{FeddersCrit}, we need to show that $f^{p-1} \in \fm^{[p]}$. A general monomial of $f^{p-1}$ is of the following form:
    $$c \cdot x_1^{i_1 \cdot d}x_2^{i_2 \cdot d} \dots x_n^{i_n \cdot d}$$
For $0 \neq c \in k$ and $\sum i_j = p-1$. As $n < \frac{p-1}{a}$, at least one of these $i_1, \dots i_n$ is at least $a+1$. Fix $j$ such that $i_j \geq a+1$. As $d \geq \left\lceil\frac{p}{a+1}\right\rceil$, this implies that $i_j \cdot d \geq p$, and so $c \cdot x_1^{i_1 \cdot d}x_2^{i_2 \cdot d} \dots x_n^{i_n \cdot d} \in \fm^{[p]}$. As this is true for every monomial, it follows that $f^{p-1} \in \fm^{[p]}$ and thus $f$ is not $F$-pure. \\

For the quasi-$F$-pure step, it is sufficient to show that $f^{p-2} \in \fm^{[p]}$ via corollary \ref{cor: QuasiFedderCriterionNotQFPure}(a). Given that $n < \frac{p-2}{a}$, the proof is identical.
\end{proof}

Combining these results yields the majority of our theorem:
\begin{theorem}\label{thm: FermatHypersurfaceClassification}
     Let $f = x_1^d + \dots + x_n^d\in A = k[x_1, \dots, x_n]$ where $k$ is of characteristic $p > 2$ and $n \geq 3$. Let $R = A/(f)$. Then for any $a \geq 1 $:
     \begin{enumerate}[label = (\arabic*)]
         \item If $\left\lceil \frac{p}{a+1}\right\rceil \leq d < \left\lceil \frac{p}{a}\right\rceil$, then we have separate cases depending on $n$:
         \begin{enumerate}[label = (\alph*)]
             \item If $n < \frac{p-2}{a}$, $R$ is not quasi-$F$-pure.
             \item If $\frac{p-2}{a} \leq n < \frac{p-1}{a}$, $R$ is not $F$-pure. Further, if $R$ is quasi-$F$-pure then $d < n$ or $d = n = 3$.  
             \item If $n \geq  \frac{p-1}{a}$, $R$ is $F$-pure.
         \end{enumerate} 
         \item If $d \geq p$, $R$ is not quasi-$F$-pure.
     \end{enumerate}
\end{theorem}
If we chase through the cases we obtain an interesting criterion for which $F$-purity and quasi-$F$-purity are equivalent notions for Fermat hypersurfaces. 
\begin{corollary}\label{cor: FermatHypQFPequivtoFP}
    If $p \not\equiv 2$ mod $n$, then any Fermat type hypersurface is quasi-$F$-pure if and only if it is $F$-pure. 
\end{corollary}
\begin{proof}
   We note that if $n \in \NN$,  $\frac{p-2}{a} \leq n < \frac{p-1}{a}$ if and only if $n = \frac{p-2}{a}$. Thus if $p > 2$ and $p \not\equiv 2$ mod $n$, case $(1b)$ in the theorem above is vacuous, and in every other case, $R$ is either $F$-pure or not quasi-$F$-pure.  If $p = 2$ then it is always equivalent to $2$ mod $n$; therefore the result follows.
\end{proof}
We now prove the theorem. 
\begin{proof}[proof of theorem \ref{thm: FermatHypersurfaceClassification}]
    The following cases are immediate given the above lemmas:
    \begin{enumerate}[label = (\arabic*)]
        \item  \begin{enumerate}[label = (\alph*)]
             \item Lemma \ref{lem: high degree}
              \addtocounter{enumii}{1}
             \item Lemma \ref{lem: low degree}
            \end{enumerate} 
        \item If $d \geq p$ then $f \in \fm^{[p]}$. As $p > 2$, it follows that $f^{p-2} \in \fm^{[p]}$, so $R$ is not quasi-$F$-pure by corollary \ref{cor: QuasiFedderCriterionNotQFPure}(a).
    \end{enumerate}
    This leaves only case $(1b)$. We note that this case is vacuous unless $p \equiv 2$ mod $n$; thus we fix $a \in \NN$ such that $n = \frac{p-2}{a}$ and $\left\lceil \frac{p}{a+1}\right\rceil \leq d < \left\lceil \frac{p}{a}\right\rceil$. We note that by lemma \ref{lem: high degree}, $R$ is not $F$-pure. \\ 
    It is thus sufficient to check that $R$ is quasi-$F$-pure, for $n,d,p$ as above, if $d < n$ or $d = n = 3$. We prove this by contrapositive. If $d > n$, $R$ is of general type and is not quasi-$F$-pure. If $d = n > 3$, as $p \not\equiv 1$ mod $n$ $R$ is never quasi-$F$-pure by \cite[Example 7.3]{kawakami2022fedder}.

\end{proof}

There is evidence to suggest that the converse of case (1b) is also true. If $n=d=3$, then $R$ is the affine cone of an elliptic curve. In this case $p \equiv 2$ mod $3 = n$, so it is known that $R$ is a supersingular elliptic curve and hence quasi-$F$-pure of height $2$. Thus, to complete the classification it is sufficient to affirmatively answer the following question:

\begin{quest}
     Let $f = x_1^d + \dots + x_n^d\in A = k[x_1, \dots, x_n]$ where $k$ is of characteristic $p > 2$ and $n \geq 3$. Suppose that $p \equiv 2$ mod $n$ and choose $a \geq 1$ such that $n = \frac{p-2}{a}$. If $\left\lceil \frac{p}{a+1}\right\rceil \leq d < n$, is $R = A/(f)$ quasi-$F$-pure?
\end{quest}
Via a standard computation written below, all Fermat type hypersurfaces being asked about in the question above have $F$-pure threshold $1 - \frac{1}{p}$. While the converse of Theorem \ref{thm: MAIN FPT/QFP} does not always hold, the following lemma suggests that the singularities in this case are well controlled.
\begin{lemma}
     Let $f = x_1^d + \dots + x_n^d\in A = k[x_1, \dots, x_n]$ where $k$ is of characteristic $p > 2$ and $n \geq 3$. Let $R = A/(f)$. Suppose that $p \equiv 2$ mod $n$ and choose $a \geq 1$ such that $n = \frac{p-2}{a}$. If $d < n$, then $\fpt(f) \geq 1 - \frac{1}{p}$. If we further assume that $d \geq \left\lceil \frac{p}{a+1}\right\rceil$, then $\fpt(f) = 1 - \frac{1}{p}$.
\end{lemma}
\begin{proof}
    We claim that for all $e \gg 0$,
    $$\nu_f(p^e) \geq p^e - 2p^{e-1} + 1$$
    If this the claim is true, then $\fpt(f) \geq 1 - \frac{2}{p} + \frac{1}{p^e}$. As $p = an + 2 > n - 2$ for $a \geq 1$, via \cite{HernandezBetancourtWittEmilyZhang_FPureThresholdsofHomogeneousPolynomials}, we can conclude that $\fpt(f) \geq 1 - \frac{h}{p}$ for some $h \geq 0$. From the above bound, it follows that $\fpt(f) \geq 1 -  \frac{1}{p}$. If $d \geq \left\lceil \frac{p}{a+1}\right\rceil$ then $R$ is not $F$-pure by Lemma \ref{lem: high degree}; it follows that $\fpt(f) = 1 - \frac{1}{p}$ in this case as desired. \\

   We now prove the claim. We've assumed that $d < n$, so it suffices to show this holds for the case $d = n- 1 = \frac{p-2}{a}-1$, as for any lower $d$ the $F$-pure threshold would only increase. It is thus sufficient to show that
    $$\left(x_1^{\frac{p-2}{a}-1} + \dots + x_{\frac{p-2}{a}}^{\frac{p-2}{a}-1} \right)^{p^e - 2p^{e-1} + 1} \not\in \fm^{[p^e]}$$
    Rewriting this polynomial as 
    $$\left(x_1^{\frac{p-2}{a}-1} + \dots + x_{\frac{p-2}{a}}^{\frac{p-2}{a}-1} \right)^{p^{e-1}(p-2)}\left(x_1^{\frac{p-2}{a}-1} + \dots + x_{\frac{p-2}{a}}^{\frac{p-2}{a}-1} \right)$$
    We see this polynomial contains the monomials
    $$M_{j} := c_j\cdot \left(x_1\cdot x_2 \cdot \dots \cdot x_n\right)^{\left(\frac{p-2}{a}-1\right)(p^{e-1})a} \cdot x_{j}^{\frac{p-2}{a}-1}$$
    For $0 \neq c_j \in k$ and each $j = 1, \dots, n = \frac{p-2}{a}$. For simplicity we look at $M_1$. Note that $\deg_{x_1}(M_1) = \deg_{x_j}(M_1) + \frac{p-2}{a}- 1$ for each $j > 1$, so it is sufficient to show that $\deg_{x_1}(M_1) < p^e$ to conclude that $M_1 \not\in \fm^{[p^e]}$, which would prove the claim. 
    \begin{flalign*}
        \deg_{x_1}(M_1) &= \left(\frac{p-2}{a}-1\right)(p^{e-1})a + \frac{p-2}{a}-1 \\
        &= p^{e-1}(p-2-a) + \frac{p-2}{a}-1 \\
        &= p^e - (a+2)p^{e-1} + \frac{p-2}{a}-1
    \end{flalign*}
    For any $e$ chosen sufficiently large (possibly dependent on $a$) and fixed $a \geq 1$ , this is less than $p^e$. 
\end{proof}
The first example of such an $f$ is of the form $(n,d,p) = (5,4,7)$. A Julia computation via the methods of \cite{BatubaraGarzellaPan_K3SurfacesofanyArtinMazurHeightoverF5andF7} shows that this is quasi-$F$-pure:
\begin{exam}
    \label{exam:fermat7:qfs}
    The fermat quartic threefold over \(\mathbb{F}_{7}\) has quasi-\(F\)-pure
    height 2.
\end{exam}
The next most simple examples exist in characteristic $11$, specifically 
$$(n,d,p) = (9,6,11),(9,7,11),(9,8,11)$$
but we are unable to compute the quasi-$F$-pure height of these examples due to limits on computing capacity.

\subsection{Questions on Quartic Threefolds}

Example \ref{exam:fermat7:qfs} is interesting in that it is an ``unlikely intersection'' of the equations defining the non-$F$-split locus. We now explain this more precisely. \\

\newcommand{\wics}{\operatorname{wics}}

We conjecture the type of each of the unknown cases for $n=5$, $d=4$. To do this, we recall a few facts about the moduli of hypersurfaces. One may compute a presentation for the moduli of projective hypersurfaces
using a general recipe.
\begin{definition}
    A \textit{weak integer composition}
    of length \(n\) and degree \(d\) is a
    \(n\)-tuple \((a_{1}, \ldots, a_{n})\) 
    of nonnegative integers
    such that \(a_{1} + \ldots + a_{n} = d\).
    We denote the set of weak integer compositions
    of length \(n\) and degree \(d\) by 
    $$ \wics(n,d) := \left\{(a_1, \dots, a_n) \in \ZZ_{\geq 0}^n \ \bigg| \ a_1 + \dots + a_n = d\right\}$$ 
\end{definition}
The set \(\wics(n,d)\) is in bijection with the set of 
homogeneous monomials of degree \(d\) in \(n\) variables,
by raising the \(i\)-th variable to the \(a_{i}\)-th power. To construct the moduli of hypersurfaces
in \(\mathbb{P}^{n-1}\) and degree \(d\),
we may consider the projective space parameterizing the 
coefficients of a generic homogeneous polynomial of
degree \(d\) in \(n\) variables, i.e. 
\(\mathbb{P}^{\#\wics(n,d)}_{k}\).
Then, we remove the locus cut out by the Jacobian ideal
to get a space parameterizing smooth hypersurfaces. 
Finally, we quotient by the group \(\PGL_{n}\) of
changes of variables.
As \(\PGL_{n}\) has dimension \(n^{2}-1\) 
and \(\#\wics(n,d) = \binom{d + n - 1}{n-1}\),
we see that the dimension of this moduli space
is  \(\binom{d + n - 1}{n-1} - n^2 + 1\). \\

We can also determine the non-\(F\)-split locus:
if one takes general polynomial 
(in $\binom{d + n - 1}{n-1}$ variables),
raises it to the \(p-1\)-th power,
then by Fedder's criterion 
the non-\(F\)-split locus is given by the coefficients
of the monomials which do not have a \(p\)-th power
of any variable.

\begin{defn}
    Let \(\wics(n,d)_{<k}\) denote the 
    subset of \(\wics(n,d)\) for which all
    of the parts \(a_{1}, \ldots, a_{n}\) in the weak integer composition are
    strictly less than \(k\).
\end{defn}

The previous discussion gives

\begin{prop}
    The non-\(F\)-split locus is cut out by 
    \(\#\wics(n,d(p-1))_{<p}\) equations.
\end{prop}

The authors do not know a combinatorial expression for the quantity
\(\#\wics(n,d(p-1))_{<p}\), but it is easily computable for small inputs. In the case of a quartic threefold (\(n = 5, d = 4\)), we have
that the dimension of the moduli space is 46.
In this case, if \(11 \leq p\), then by
\cite[Theorem~3.12]{HernandezBetancourtWittEmilyZhang_FPureThresholdsofHomogeneousPolynomials} all quartic threefolds are \(F\)-split.
For \(p=2,3,5\) and \(7\), the number of equations
that cut out the non-\(F\)-split locus is \(5, 15, 70,\) and \(210\) 
respectively. \\

Notice that for \(p=5,7\), the number of equations cutting out
the non-\(F\)-split locus is larger than the dimension of the moduli
space. Thus, one may expect that all such hypersurfaces are \(F\)-split.
However, this turns out to be too much to hope for: Example \ref{exam:fermat7:qfs} is an ``unlikely intersection'' of the equations
which define the non-\(F\)-split locus. This contrasts the \(p=2\) or \(3\) case where one knows that non-\(F\)-split surfaces exist for moduli reasons (i.e. the non-\(F\)-split locus is not cut out by enough equations to vanish completely). \\

Intuitively, it's possible for such a hypersurface in characteristic \(7\) 
to be non-\(F\)-split because each monomial is very ``tall'' (i.e. each monomial has a single variable of a high power) rather than ``wide'' (i.e. a monomial having low powers of many variables) However, the Fedder-type criterion for quasi-\(F\)-purity ``mixes up'' the variables. Motivated by this intuition we suggest the following question:

\begin{quest}
    Let \(X\) be a projective hypersurface of degree
    \(d\) in \(\mathbb{P}^{n-1}\) over \(\mathbb{F}_{p}\).
    Assume that 
    \[\binom{d + n - 1}{n-1} - n^2 + 1 < \#\wics(n,d(p-1))_{<p}\]
    Then is \(X\) always quasi-\(F\)-split?
\end{quest}
If true, it would imply the following corollary:
\begin{cor}
    Let \(X\) be a quartic threefold over \(\mathbb{F}_{5}\) or \(\mathbb{F}_{7}\).
    Then \(X\) is quasi-\(F\)-split.
\end{cor}
The authors do not know of a single example
of a quartic threefold in characteristic \(5\) which is not
\(F\)-split. In contrast, by corollary \ref{cor: QuasiFedderCriterionNotQFPure}, the Fermat quartic threefold over \(\mathbb{F}_{3}\) is not quasi-$F$-split. 



\section{Competing Interests}

The authors declare none.

\bibliographystyle{amsalpha}
\bibliography{references,preprints,MainBib}

\newcommand{\etalchar}[1]{$^{#1}$}
\def\cfudot#1{\ifmmode\setbox7\hbox{$\accent"5E#1$}\else \setbox7\hbox{\accent"5E#1}\penalty 10000\relax\fi\raise 1\ht7 \hbox{\raise.1ex\hbox to 1\wd7{\hss.\hss}}\penalty 10000 \hskip-1\wd7\penalty 10000\box7}
\providecommand{\bysame}{\leavevmode\hbox to3em{\hrulefill}\thinspace}
\providecommand{\MR}{\relax\ifhmode\unskip\space\fi MR }
\providecommand{\MRhref}[2]{%
  \href{http://www.ams.org/mathscinet-getitem?mr=#1}{#2}
}
\providecommand{\href}[2]{#2}
\begin{thebibliography}{HNnBWZ16}

\bibitem[BGP25]{BatubaraGarzellaPan_K3SurfacesofanyArtinMazurHeightoverF5andF7}
Ryan Batubara, Jack~J Garzella, and Alex Pan, \emph{K3 surfaces of any artin-mazur height over $\mathbb{F}_5$ and $\mathbb{F}_7$ via quasi-$f$-split singularities and gpu acceleration}, 2025.

\bibitem[BMS06]{BudurMustataSaitoBernsteinSatoPolynomialsOfArbitrary}
Nero Budur, Mircea Musta{\c{t}}{\v{a}}, and Morihiko Saito, \emph{Bernstein-{S}ato polynomials of arbitrary varieties}, Compos. Math. \textbf{142} (2006), no.~3, 779--797. \MR{MR2231202 (2007c:32036)}

\bibitem[BMS08]{BlickleMustataSmithDiscretenessAndRationalityOfFThresholds}
Manuel Blickle, Mircea Musta{\c{t}}{\v{a}}, and Karen~E. Smith, \emph{Discreteness and rationality of {$F$}-thresholds}, Michigan Math. J. \textbf{57} (2008), 43--61, Special volume in honor of Melvin Hochster. \MR{2492440 (2010c:13003)}

\bibitem[BMS09]{BlickleMustataSmithFThresholdsOfHypersurfaces}
Manuel Blickle, Mircea Musta{\c{t}}{\u{a}}, and Karen~E. Smith, \emph{{$F$}-thresholds of hypersurfaces}, Trans. Amer. Math. Soc. \textbf{361} (2009), no.~12, 6549--6565. \MR{2538604 (2011a:13006)}

\bibitem[Bor11]{Borger}
James Borger, \emph{The basic geometry of {W}itt vectors, {I}: {T}he affine case}, Algebra Number Theory \textbf{5} (2011), no.~2, 231--285. \MR{2833791}

\bibitem[BS15]{BhattSinghThresholds}
Bhargav Bhatt and Anurag~K. Singh, \emph{The {$F$}-pure threshold of a {C}alabi-{Y}au hypersurface}, Math. Ann. \textbf{362} (2015), no.~1-2, 551--567. \MR{3343889}

\bibitem[DK14]{DavisKedlaya}
Christopher Davis and Kiran~S. Kedlaya, \emph{On the {W}itt vector {F}robenius}, Proc. Amer. Math. Soc. \textbf{142} (2014), no.~7, 2211--2226. \MR{3195748}

\bibitem[DM23]{DattaMurayamaTate}
Rankeya Datta and Takumi Murayama, \emph{Tate algebras and {F}robenius non-splitting of excellent regular rings}, J. Eur. Math. Soc. (JEMS) \textbf{25} (2023), no.~11, 4291--4314. \MR{4662293}

\bibitem[DN15]{DeStefaniNunezBetancourtFthresholdGradedrings}
Alessandro {De Stefani} and Luis {N\'{u}\~{n}ez-}{Betancourt}, \emph{{$F$}-threshold of graded rings}, arXiv: 1507.05459.

\bibitem[Fed83]{FeddersCrit}
Richard Fedder, \emph{{$F$}-purity and rational singularity}, Trans. Amer. Math. Soc. \textbf{278} (1983), no.~2, 461--480. \MR{701505}

\bibitem[FGM{\etalchar{+}}]{fgmqt-2025-witt-vectors-macaulay2}
Anne Fayolle, Abhay Goel, Devlin Mallory, Eamon Quinlan-Gallego, and Teppei Takamatsu, \emph{The witt vectors package for {Macaualay2}}, In preparation.

\bibitem[GS]{M2}
Daniel~R. Grayson and Michael~E. Stillman, \emph{Macaulay2, a software system for research in algebraic geometry}, Available at \url{http://www2.macaulay2.com}.

\bibitem[Her11a]{HernandezFPureThresholdOfBinomial}
Daniel Hern{\'a}ndez, \emph{F-pure thresholds of binomial hypersurfaces}, arXiv:1112.2427.

\bibitem[Her11b]{HernandezThesis}
Daniel Hern\'andez, \emph{$f$-purity for hypersurfaces}, Ph.D. thesis, University of Michigan, 2011.

\bibitem[Her15]{HernandezFInvariantsOfDiagonalHyp}
Daniel~J. Hern\'{a}ndez, \emph{{$F$}-invariants of diagonal hypersurfaces}, Proc. Amer. Math. Soc. \textbf{143} (2015), no.~1, 87--104. \MR{3272734}

\bibitem[HMTW08]{HunekeMustataTakagiWatanabeFThresholdsTightClosureIntClosureMultBounds}
Craig Huneke, Mircea Musta{\c{t}}{\u{a}}, Shunsuke Takagi, and Kei-ichi Watanabe, \emph{F-thresholds, tight closure, integral closure, and multiplicity bounds}, Michigan Math. J. \textbf{57} (2008), 463--483, Special volume in honor of Melvin Hochster. \MR{MR2492463}

\bibitem[HNnBWZ16]{HernandezBetancourtWittEmilyZhang_FPureThresholdsofHomogeneousPolynomials}
Daniel~J. Hern\'andez, Luis N\'u\~nez Betancourt, Emily~E. Witt, and Wenliang Zhang, \emph{{$F$}-pure thresholds of homogeneous polynomials}, Michigan Math. J. \textbf{65} (2016), no.~1, 57--87. \MR{3466816}

\bibitem[HSTW21]{FrobeniusThresholdsMC2}
Daniel~J. Hern\'andez, Karl Schwede, Pedro Teixeira, and Emily~E. Witt, \emph{The {F}robenius{T}hresholds package for {M}acaulay2}, J. Softw. Algebra Geom. \textbf{11} (2021), no.~1, 25--39. \MR{4285762}

\bibitem[HW02]{HaraWatanabeFRegFPure}
Nobuo Hara and Kei-Ichi Watanabe, \emph{F-regular and {F}-pure rings vs. log terminal and log canonical singularities}, J. Algebraic Geom. \textbf{11} (2002), no.~2, 363--392. \MR{MR1874118 (2002k:13009)}

\bibitem[Jag26]{Jagathese_OnQuasiFPurityofExcellentRings}
Vignesh Jagathese, \emph{On quasi-{$F$}-purity of excellent rings}, Comm. Algebra \textbf{54} (2026), no.~4, 1760--1771.

\bibitem[KKP{\etalchar{+}}22]{KadyrsizovaKenkelPageSinghSmithVraciuWittLowerBoundsonFPureThresholdsandExtremalSingularities}
Zhibek Kadyrsizova, Jennifer Kenkel, Janet Page, Jyoti Singh, Karen~E. Smith, Adela Vraciu, and Emily~E. Witt, \emph{Lower bounds on the {$F$}-pure threshold and extremal singularities}, Trans. Amer. Math. Soc. Ser. B \textbf{9} (2022), 977--1005. \MR{4498775}

\bibitem[KTT{\etalchar{+}}22]{Kawakami++_2022quasifsplittingsbirationalgeometry}
Tatsuro Kawakami, Teppei Takamatsu, Hiromu Tanaka, Jakub Witaszek, Fuetaro Yobuko, and Shou Yoshikawa, \emph{Quasi-f-splittings in birational geometry}.

\bibitem[KTT{\etalchar{+}}24a]{Kawakami++_2022quasifsplittingsbirationalgeometryII}
Tatsuro Kawakami, Teppei Takamatsu, Hiromu Tanaka, Jakub Witaszek, Fuetaro Yobuko, and Shou Yoshikawa, \emph{Quasi-{$F$}-splittings in birational geometry {II}}, Proc. Lond. Math. Soc. (3) \textbf{128} (2024), no.~4, Paper No. e12593, 81. \MR{4731853}

\bibitem[KTT{\etalchar{+}}24b]{Kawakami++_2022quasifsplittingsbirationalgeometryIII}
Tatsuro Kawakami, Teppei Takamatsu, Hiromu Tanaka, Jakub Witaszek, Fuetaro Yobuko, and Shou Yoshikawa, \emph{Quasi-{$F$}-splittings in birational geometry iii}.

\bibitem[KTY22]{kawakami2022fedder}
Tatsuro {Kawakami}, Teppei {Takamatsu}, and Shou {Yoshikawa}, \emph{{Fedder type criteria for quasi-$F$-splitting}}, arXiv e-prints (2022), arXiv:2204.10076.

\bibitem[Len19]{LenstraWittVectors}
Hendrik~W. Lenstra, \emph{Construction of the ring of {W}itt vectors}, Eur. J. Math. \textbf{5} (2019), no.~4, 1234--1241. \MR{4015455}

\bibitem[LZ04]{LangerZink}
Andreas Langer and Thomas Zink, \emph{De {R}ham-{W}itt cohomology for a proper and smooth morphism}, J. Inst. Math. Jussieu \textbf{3} (2004), no.~2, 231--314. \MR{2055710}

\bibitem[MTW05]{MustataTakagiWatanabeFThresholdsAndBernsteinSato}
Mircea Musta{\c{t}}{\v{a}}, Shunsuke Takagi, and Kei-ichi Watanabe, \emph{F-thresholds and {B}ernstein-{S}ato polynomials}, European Congress of Mathematics, Eur. Math. Soc., Z\"urich, 2005, pp.~341--364. \MR{MR2185754 (2007b:13010)}

\bibitem[NY21]{NakkajimaYobukoKodaira}
Yukiyoshi Nakkajima and Fuetaro Yobuko, \emph{Degenerations of log {H}odge de {R}ham spectral sequences, log {K}odaira vanishing theorem in characteristic {$p>0$} and log weak {L}efschetz conjecture for log crystalline cohomologies}, Eur. J. Math. \textbf{7} (2021), no.~4, 1537--1615. \MR{4340947}

\bibitem[Pet25]{Petrov_DecompositionofDeRhamComplexForQuasiFSplitVarieties}
Alexander Petrov, \emph{Decomposition of de rham complex for quasi-f-split varieties}.

\bibitem[TW04]{TakagiWatanabe_OnFPureThresholds}
Shunsuke Takagi and Kei-ichi Watanabe, \emph{On {F}-pure thresholds}, J. Algebra \textbf{282} (2004), no.~1, 278--297. \MR{2097584}

\bibitem[TWY24]{tanaka2024quasifesplittings}
Hiromu Tanaka, Jakub Witaszek, and Fuetaro Yobuko, \emph{Quasi-$f^e$-splittings and quasi-$f$-regularity}.

\bibitem[vdGK02]{vdG-Katsura-2002-invariant}
G.~van~der Geer and T.~Katsura, \emph{An invariant for varieties in positive characteristic}, Algebraic number theory and algebraic geometry, Contemp. Math., vol. 300, Amer. Math. Soc., Providence, RI, 2002, pp.~131--141. \MR{1936370}

\bibitem[Yob19]{Yobuko_QuasiFrobeniusSplittingandliftingofCYVarsInCharp}
Fuetaro Yobuko, \emph{Quasi-{F}robenius splitting and lifting of {C}alabi-{Y}au varieties in characteristic {$p$}}, Math. Z. \textbf{292} (2019), no.~1-2, 307--316. \MR{3968903}

\bibitem[Yob20]{Yobuko_OnTheFrobeniusSplittingHeightofVarsinPositiveChar}
\bysame, \emph{On the {F}robenius-splitting height of varieties in positive characteristic}, Algebraic number theory and related topics 2016, RIMS K\^oky\^uroku Bessatsu, vol. B77, Res. Inst. Math. Sci. (RIMS), Kyoto, 2020, pp.~159--175. \MR{4278982}

\bibitem[Yob23]{Yobuko_QuasiFSplitandHodgeWitt}
Fuetaro Yobuko, \emph{Quasi-f-split and hodge-witt}, 2023.

\bibitem[Yos25]{yoshikawa2025feddertypecriterionquasifesplittingquasifregularity}
Shou Yoshikawa, \emph{Fedder-type criterion for quasi-$f^e$-splitting and quasi-$f$-regularity}, 2025.

\end{thebibliography}

\end{document}